\documentclass[12pt]{amsart}
\usepackage{latexsym}
\usepackage{amssymb}
\usepackage{amsthm}
\usepackage{latexsym}
\usepackage{amsmath}

\def\gp#1{\langle #1 \rangle}

\newtheorem{theorem}{Theorem}
\newtheorem{lemma}[subsection]{Lemma}
\newtheorem{corollary}[subsection]{Corollary}

\title[$d$-generator property]{Modules over some group rings, having $d$-generator property}

\thanks{The  research was supported by the UAEU UPAR Grant G00002160}

\author[Bovdi]{V.A.~ Bovdi}

\address{UAEU, Al-Ain, United Arab Emirates}
\email{vbovdi@gmail.com}

\author[Kurdachenko]{L.A.~Kurdachenko}

\address{ Department of Algebra and Geometry,
School of Mathematics and Mechanics,
National University of Dnipro, Ukraine}
\email{lkurdachenko@i.ua}



\begin{document}

\maketitle

\begin{abstract}
For modules over group rings we introduce  the following numerical parameter. We say that a module $A$ over a ring $R$ has finite $r$-generator property if each f.g. (finitely generated) $R$-submodule of $A$  can be generated exactly by $r$  elements and there exists a f.g. $R$-submodule $D$  of $A$, which has a minimal generating subset, consisting exactly of $r$  elements.
Let $FG$  be the group algebra of a finite group $G$ over  a field  $F$. In the present paper modules over the  algebra  $FG$  having finite generator property are described.
\end{abstract}

\section{Introduction and main results}

Let $R$ be an associate  ring and let $A$  be a module over the ring $R$. The  structure and the properties of $A$ essentially depend on the properties of the  family of its finitely generated submodules. Clearly, if $B$ is a finitely generated submodule of $A$ then   $B$  can have different generator systems. Therefore, it makes sense to consider a minimal system of generators. Recall that a subset $M\subseteq B$  is called {\it a minimal generating subset} for $B$ if $M$ generates $B$ and each  proper subset of $M$ generates a proper submodule of $B$. If $R$ is a field and $A$ is a vector space, then the number of elements in each minimal system of generators for  $B$  is the same, i.e. it is an invariant of  $B$. Generally  such a situation does not always take place.  Moreover, there are not so many rings for which this holds. Therefore, in the submodule  $B$, we consider those minimal systems of generators that have the least number of elements. The number of elements in such minimal system of generators is denoted by $d_R(A)$.

If $A$ is a vector space over a field $F$, then $d_F(A)$ is exactly the dimension of $A$. However, when moving from vector spaces to modules, the properties of the numerical invariant $d_R (A)$ can already differ significantly from the classical properties of the dimensions of vector subspaces. For example, if $B$ is a proper subspace of the the  vector space $A$, then  $\dim_F (B)<  \dim_F (A)$. Unlike vector spaces, not each submodule of a finitely generated module is finitely generated. Moreover, in the case when a submodule is finitely generated it can has a minimal generating set with more elements than the number of elements in the generating system of the entire module. The following simple example illustrates this situation.

Let $p$  be  a prime. Let  $G = \gp{g}\times \gp{h}\cong C_p\times C_p$ be a  multiplicative $p$-group of order  $p^2$,  $A = \gp{b}\oplus \gp{c}\oplus  \gp{d}\cong \mathbb{Z}_p\oplus \mathbb{Z}_p\oplus \mathbb{Z}_p$ an  additive elementary abelian  $p$-group of order $p^3$ and  let $H=\gp{c}\oplus  \gp{d}\leq A$.   Define a right   action  $\circ$ of  $g$ and $h$   on  $A$  (concerning  the   basis  $\{ b, c, d\}$)   by the following rule:
\[
a\circ g=
\begin{pmatrix}
 b\\ c\\ d\\
\end{pmatrix}
\begin{pmatrix}
1&  0 & 0\\
1&  1 & 0 \\
0 & 0 & 1\\
\end{pmatrix}
\quad\text{and}\quad
a\circ h=\begin{pmatrix}
 b\\ c\\ d\\
\end{pmatrix}
\begin{pmatrix}
1&  0 & 0\\
0&  1 & 0 \\
1 & 0 & 1\\
\end{pmatrix},
\qquad (a\in A).
\]
In other  words  $b\circ  g = b + c$, $c\circ g = c$, $d\circ g = d$, $b\circ  h = b + d$, $c\circ h = c$ and  $d\circ h = d$. Clearly,   $A = \gp{b}_{\mathbb{Z}_pG}$  is a  cyclic right  $\mathbb{Z}_pG$-module, in which  $d_{\mathbb{Z}_pG}(A) = 1$, but $d_{\mathbb{Z}_pG}(H) = 2$.

Each ring $R$ and their ideals are itself the modules over $R$. Therefore, it is natural to consider first the number of generating elements (left or right) ideals in the ring $R$. It is clear that it makes sense to do this for the rings whose (left or right) ideals are finitely generated, i.e. for the Noetherian rings. In \cite{CI1950} those commutative rings  $R$ were considered, for which there exists a positive integer $r$ such that $d_R (I)\leq r$  for each ideal $I$ of  $R$. It is immediately clear that all examples of such rings are principal ideal domains and Dedekind domains. Commutative rings with such  property actively  studied in \cite{CI1950, GR1969, GR1972, GR1973, Heitmann_1976, Matson_2009, Pettersson_1999, Swan_1984}.  Gilmer \cite{GR1969}  generalized this situation, where he considered  commutative rings for which there exists a positive integer $r$ such that $d_R (I) \leq r$  for every finitely generated ideal $I$ of $R$. Such rings are called {\it rings with $r$ - generator property}. The most famous examples of such rings (see, for example, \cite[Chapter II.1, Chapter III.5]{FS2001}) are  valuation domains and   Bezout domains (i.e. domains in which every finitely generated ideal  is principal). Commutative rings having  $r$-generator property have been studied in  \cite{GR1969, GR1972, GR1973, Heitmann_1976, Matson_2009, Pettersson_1999, Swan_1984}.
In  group theory, an analogue of this concept arose earlier in the work of Maltsev \cite{MAL1948}. Following Maltsev,  we say that a group $G$  has finite special rank $r$, where $r$ is a positive integer, if $d(H)\leq r$  for every finitely generated subgroup $H$ of the  group $G$.
In group theory this concept turned out to be very useful and productive (for example, see the  survey \cite{Giovanni_2013} and the  book \cite{DKS2017}).
Thus, it is a native  to consider an analog of this natural  concept for modules as well.
Let $A$ be a module over a ring $R$. We say that $A$ has {\it $d$-$R$-generator property} if there exists a positive integer $d$ such that $d_R (B)\leq d$ for each finitely generated $R$-submodule $B$ of $A$. If we consider only one fixed ring $R$, then we  simply say that $A$ has $d$-generator property.

In \cite{Bovdi_Kurdachenko} modules were studied which have $d$-generator property over the   group algebra $FG$,    where $F$  is a field of characteristic 0 and $G$ is a finite group. In the paper mentioned above  we used another term - "special rank", but in ring theory there are already several terms using the word rank, so in this paper  we preferred to use the term introduced by R.~Gilmer for the rings.

In the present  paper we study  modules having $d$-generator property over the group  algebras  $FG$  where $F$  is a field of arbitrary characteristic and $G$ is a finite group. The main results are the following.

\begin{theorem}\label{Th:1}
Let  $r\in\mathbb{N}$ and let  $FG$ be  the  group algebra of a finite group $G$ over  a field $F$.
An  $FG$-module $A$  has a finite   $r$-generator property  if and only if  the following conditions hold:
\begin{itemize}
\item[(i)]   $\dim_F(A)$ is  finite;

\item[(ii)] each  semisimple homogeneous  $FG$-factor of  $A$  has width at most  $r$  and there exists a semisimple homogeneous  $FG$-factor, whose width is exactly  $r$.
\end{itemize}
\end{theorem}

\begin{theorem}\label{Th:2}
Let $n,r\in\mathbb{N}$ and  let  $FG$ be  the  semisimple group algebra of a finite group $G$  over  a field $F$ of prime characteristic  $p$. An  $FG$-module $A$  has  finite $r$-generator property  if and only if  the following conditions hold:
\begin{itemize}
\item[(i)]  $A = \oplus_{j=1}^n H_j$ in which each  $H_j = \oplus_{k=1}^{t(j)} D_k$, where each  $D_k$  is a simple  $FG$-submodule of  $A$, $|G|\leq t(j)$ and  $D_k \cong_{FG} D_m$  for all  $1 \leq k, m \leq t(j)$;
\item[(ii)]  $t(j) \leq r$  for all  $j$  and there exists   $s\in\mathbb{N}$  such that  $t(s) = r$;
\item[(iii)] $n \leq nns_F(G)$, where $nns_F(G)$  is the number of pairwise non-isomorphic simple  $FG$-modules.
\end{itemize}
\end{theorem}

\begin{theorem}\label{Th:3}
Let $m,r\in\mathbb{N}$ and  let   $FG$ be  the modular group algebra of a finite group $G$ over  a field $F$ of prime characteristic  $p$.  Let $P\in Syl_p(G)$  be a normal Sylow  $p$-subgroup of  $G$. If  an  $FG$-module $A$  has finite $r$-generator property,  then  the following conditions hold:
\begin{itemize}
\item[(i)] $\dim_F(A)$ is  finite;
\item[(ii)] $A$   has a series of  $FG$-submodules
\[
\gp{0} = Z_0 \leq   Z_1 \leq  \cdots\leq  Z_{m-1} \leq  Z_m = A
\]
such that  $m \leq  |P|$  and each  $G/C_G(Z_{j + 1} /Z_j)$  is a   $p$-group,
each factor  $Z_{j + 1} /Z_j$   is a semisimple  $FG$-module, whose homogeneous components have width at most  $r$  and the number of homogeneous components is at most  $nns_F(G/P)$, where $nns_F(G)$  is the number of pairwise non-isomorphic simple  $FG$-modules.
\end{itemize}
\end{theorem}

\section{Preliminaries and Lemmas}
In the sequel,  we use  the notions and results from the books \cite{Bovdi_book, Curtis_Reiner, DKS2017}.

\begin{lemma}\label{LeM:1}
Let $A$  be an  $R$-module over a ring $R$ such that   $A$ has  finite $r$-generator property. If  $B, C$  are $R$-submodules of $A$  such that $B\leq_R C$ then the following conditions hold:
\begin{itemize}
\item[(i)] the submodule  $B$ has finite $r$-generator property at most $r$;
\item[(ii)] the factor-module  $A/B$ has finite $r$-generator property at most $r$;
\item[(iii)] the factor-module $C/B$ has finite $r$-generator property at most $r$.
\end{itemize}

\end{lemma}

These assertions are obvious.

\begin{lemma}\label{LeM:2}
Let $r,t\in\mathbb{N}$ and let $B$   be a submodule of an  $R$-module $A$ over a Noetherian ring $R$.  If     $B$ has  finite $r$-generator property at most $r$   and  $A/B$ has  finite $d$-generator property at most $t$, then   $A$ has  $d$-generator property  at most  $r+t$.
\end{lemma}

\begin{proof}
Let  $D$  be a f.g.   $R$-submodule of $A$.  Clearly,  $D$ is Noetherian  as an f.g. module (see \cite[Lemma 1.1]{KOS2007}). The factor  $D/(D \cap  B) \cong  (D + B)/B$  has  finite generating subset  $\{ d_j + (D \cap  B)  \mid  1 \leq j \leq m \}$ in which   $m \leq t$. It follows that   $D \cap  B$  is a f.g.   $R$-submodule of  $B$. Since  $B$  has  $r$-generator property,  $D \cap  B$  has  finite generating subset $\{ b_j  \mid  1 \leq j \leq k \}$  where  $k \leq r$. The  subset    $\{ d_j  \mid  1 \leq j \leq m \}\cup \{ b_j  \mid  1 \leq j \leq k \}$ generates $D$  as an  $R$-submodule and   $m + k \leq t + r$.
\end{proof}

\begin{lemma}\label{LeM:3}
Let  $A =\oplus_{j=1}^{n} A_j$    be an  $R$-module over a ring $R$ in which  each   $A_j$  is a simple  $R$-submodule.  If elements of the set $\{ Ann_R^{left}(A_1),\ldots,  Ann_R^{left}(A_n)\}$ are pairwise non-isomorphic,  then  $A$  is a cyclic  $R$-module.
\end{lemma}

\begin{proof} Let us prove that $A=\gp{a_1 + \cdots+ a_n \mid 0\not= a_j\in A_j }_R$. Since  each $A_j$  is a simple   $R$-module,  $U_j: = Ann_R^{left}(A_j)$  is a maximal left ideal of  $R$. The fact that a sum of left ideals is also a left ideal together with  $U_j \not= U_m $ implies that  $U_j + U_m = R$  for all  $j \not= m$. Thus  $U_1 + U_2 = R$ and  there exists   $\alpha  \in  U_2\setminus   U_1$ such that    $\alpha a_1 \not= 0$,  $\alpha a_2 = 0$  and
\[
\alpha (a_1 + \cdots+ a_n) =  \alpha a_1 + (\alpha a_3 + \cdots+ \alpha a_n).
\]
If  $\alpha a_3 + \cdots+ \alpha a_n = 0$, then  $\alpha a_1 \in  R(a_1 + \cdots+ a_n)$.  Since  $A_1$  is a simple  $R$-module, it  can be  generated by each of  its non-zero elements, so $A_1 = \gp{a_1}_R \leq \gp{a_1 + \cdots+ a_n}_R$.

Let  $\alpha a_3 + \cdots+ \alpha a_n \not= 0$ and  let  $j_1\in \mathbb{N}$  be a smallest integer (say  $j_1 = 3$)  such that  $j_1 > 2$  and  $\alpha A_{j_1} \not= 0$. The  equality $U_1 + U_3 = R$ shows that  there exists     $\beta \in  U_3\setminus   U_1$ such that   $\beta\alpha a_1 \not= 0$,  $\beta\alpha a_3 = 0$  and
\[
\beta(\alpha a_1 + \alpha a_3 + \cdots+ \alpha an) = \beta\alpha a_1 + (\beta\alpha a_4 + \cdots+ \beta\alpha a_n).
\]
If  $\beta\alpha a_4 + \cdots+ \beta\alpha a_n = 0$, then,  as above, $A_1 \leq \gp{a_1 + \cdots+ a_n}_R$.  Otherwise, all this process we can repeat again. Since $n\in\mathbb{N}$,   after finitely many steps we obtain that  $A_1 \leq R(a_1 + \cdots+ a_n)$ and $A = A_1 \oplus  \cdots\oplus  A_n = \gp{a_1 + \cdots+ a_n}_R$.
 \end{proof}

\begin{lemma}\label{LeM:4}
Let  $A =\oplus_{j=1}^{n} A_j $  be an  $R$-module over a ring  $R$ in which    each $A_j$  is a simple  $R$-submodule. If all elements of the set $\{ Ann_R^{left}(A_1),\ldots,  Ann_R^{left}(A_n)\}$ coincide,   then  $Ann_R^{left}(A_j)=Ann_R^{left}(\gp{a}_R)$ for each  $a \in  A\setminus\{0\}$. In particular, the  $R$ -submodule $\gp{a}$ is isomorphic to each  $A_j$ for  $1 \leq j \leq n$.
\end{lemma}

\begin{proof}
Set $L:=Ann_R^{left}(A_1)$. Clearly,   $L \leq  Ann_R^{left}(a)$  for each $0\not=a \in  A$, so that  $L \leq  Ann_R^{left}(A)$. Since  $A_1$  is a simple  $R$-module, $L$  is a maximal left ideal of  $R$. It follows that  $L = Ann_R^{left}(A)$.

Let $y\in R\setminus L$ and let    $a=a_1 + \cdots+ a_n\in A\setminus\{0\}$,    in which   each $a_j \in  A_j$,   $a_1 \not= 0$ and  $Ann_R^{left}(a) = R$. Since $ya=0$ and $a_1,  \ldots, a_n$ are linearly independent over $R$, $yA_j = 0$  for each $j$. In particular,    $ya_1 = 0$   and  $y \in  Ann_R^{left}(A_1) = L$,  a contradiction. It follows that    $Ann_R^{left}(a) \not= R$  for each  $a \in  R\setminus\{0\}$. Consequently,   $Ann_R^{left}(a) = L$ and   $\gp{b}_R \cong_R R/L$  for each $b \in  R\setminus\{0\}$. \end{proof}

\begin{corollary}\label{CoR:4}
Let $\Lambda$ be a  set and let  $A =\oplus_{\lambda \in  \Lambda} A_\lambda$   be an  $R$-module over a ring $R$ in which  each
$A_\lambda$ is a simple  $R$-submodule.  If
\[
Ann_R^{left}(A_\lambda) = Ann_R^{left}(A\mu), \qquad (\lambda, \mu  \in \Lambda)
\]
then each  cyclic  $R$-submodule of  $A$  is simple and   $A_\lambda\cong_R \gp{a}_R$ for each $a \in  A\setminus \{0\}$.
\end{corollary}

\begin{corollary}\label{CoR:5}
Let $n, r\in \mathbb{N}$ and let $A =\oplus_{j=1}^{n} A_j $  be an  $R$-module over   a ring $R$  in which    each  $A_j$  is a simple $R$-submodule.

If all elements of the set $\{ Ann_R^{left}(A_1),\ldots,  Ann_R^{left}(A_n)\}$  coincide,  then the following conditions are equivalent:
\begin{itemize}
\item[(i)] $d_R(A) = r$;
\item[(ii)]   $A$  has  finite $r$-generator property;
\item[(iii)]  $n=r$, where $n={\rm width}_R(A)$.
\end{itemize}
\end{corollary}

\begin{proof} $(i)\Leftrightarrow (iii)$.
Let  $S:= \{ b_1, \ldots, b_k \}\subseteq A$  be a minimal set,  such that   $A=\gp{S}_R$  and let $B_j:=\gp{b_j}_R$.  Obviously,   $B_j\cong_R  A_j$ after a suitable renumbering of the set $\{A_k\}$ and $B_j$ is simple  for each $j$ by  Lemma  \ref{LeM:4}.
Since  $A$  is a semisimple  $R$-module,  $A$  includes an  $R$-submodule  $D_1$  such that  $A = B_1 \oplus  D_1$
(see \cite[Corollary 4.3]{KOS2007}).  If   $b_t\in  B_1$ for some $t>1$, then  $A=\gp{S\setminus  b_t}_R$,  a contradiction. This yields  that  $b_t \not\in  B_1$. Similarly, if $t>2$ then $b_t \not\in  B_1\oplus B_2$. Repeating this process finitely many times, we obtain $b_k \not\in  B_1\oplus\cdots \oplus  B_{k-1}$ and   $A = \oplus_{j=1}^k   B_j$.

If  $T:=\{ d_1, \ldots, d_m \}\subseteq A$,  such that $S\not=T$ and $A=\gp{T}_R$, then repeating the above arguments, we obtain a  decomposition  $A = \oplus_{j=1}^m D_j$ in which   each $D_j=\gp{d_j}_R\cong_R A_j$ after a suitable renumbering of the set $\{A_k\}$. It follows that   $k = m$ by  the Krull-Remak-Schmidt theorem (see \cite[Chapter 6, Theorem 1.6]{CP1991})  and   $d_R(A) = n$.

$(ii)\Leftrightarrow (iii)$. Since  $A$  is a f.g. $R$-module,  $d_R(A) = r$.  Thus  $n = r$ by the part above.

Conversely, let  $n = r$. Using part (i)  we obtain that  $d_R(A) = r$. Let  B  be a  f.g.   $R$-submodule of  A. Since  $A$  is a semisimple   $R$-module, there exists  an  $R$-submodule  C  such that  $A=B \oplus C$ (see \cite[Corollary 4.3]{KOS2007}). Since   $B, C$  are semisimple,  $B = \oplus _{j=1}^s  B_j$ and  $C = \oplus_{j=1}^t  C_j$, where each $B_j$ and $C_j$  are simple  $R$-submodules (see \cite[Corollary 4.3]{KOS2007}). Consequently, using the Krull-Remak-Schmidt theorem, we get
\[
A = (\oplus_{j=1}^s B_j) \oplus  (\oplus_{j=1}^t  C_j),
\]
in which $s + t = r$ and  $B_j \cong_R A_1$, so $d_R(B) = s \leq  r$ by  part (i).  \end{proof}

Let $A =\oplus_{j=1}^{n} A_j $ be an  $R$-module over a ring $R$ in which    each $A_j$  is a simple  $R$-module. The relation  "$A_j \cong_R A_m$"  is an equivalence relation on the set  $\{ A_j  \mid 1 \leq j \leq n \}$. Let  $E_1, \ldots, E_k$  be all equivalence classes by its relation.   A direct sum of all submodules from the  set  $E_j$ we denote by  $S_j$ and it is called a {\it homogeneous component of  A.}

\begin{corollary}\label{Cor:6} Let $n, r\in \mathbb{N}$ and let $A =\oplus_{j=1}^{n} A_j $  be a semisimple  $R$-module over a ring $R$  in which   each  $A_j$  is a simple  $R$-module. The following conditions are equivalent:
\begin{itemize}
\item[(i)]  $d_R(A) = r$;

\item[(ii)] $A$  has  $r$-generator property;

\item[(iii)]  all homogeneous components of  $A$  have  width at most  $r$  and there exists a homogeneous component, having width exactly  $r$.
\end{itemize}
\end{corollary}

\begin{proof} $(i)\Leftrightarrow (iii)$. Let $d_R(A) = r$.  Let $\{E_1, \ldots, E_k\}$  be   a collection of  all homogeneous components of  $A$. Set $D_m = \oplus_{j=1, j \not= m}^n E_j$. Clearly,  $A =\oplus_{j=1}^k  E_j$ and  $A/D_m\cong_K E_m$ is a semisimple homogeneous module, so
\[
d_R(A/D_m) \leq  d_R(A) = r, \qquad (1\leq m\leq n)
\]
by Lemma \ref{LeM:1}(ii). Since $d_R(E_m)={\rm width}_R(E_m)$  by Corollary \ref{CoR:5}(i)(iii),
\[
{\rm width}_R(E_m)\leq r.
\]
Without lost  of generality we can  assume that $E_1$ has a greatest width,  say  \newline  ${\rm width}_R(E_1)=s$ with $s\in \mathbb{N}$.
Furthermore,  $E_1 = \oplus_{j=1}^s  C_j$ in which   each $C_j$  is a simple  $R$-submodule whose homogeneous components have  width  $1$ by Corollary  \ref{CoR:5}(iii).  Each  module  $C_j$  is cyclic by Lemma \ref{LeM:4}, so $A$  has a set of generators consisting of  $s$  elements. Hence  $s \leq  r$. The above proved inequality $r \leq  s$   implies that  $r = s$.
Using similar arguments,  the conditions of our lemma are sufficient.

$(ii)\Leftrightarrow (iii)$. Since  $A$  is a f.g. $R$-module, $d_R(A) = r$, so  we can repeat the proof of part $(i)\Rightarrow (iii)$.

Conversely, $(iii)\Rightarrow (i)$ as was proved  before, so   $d_R(A) = r$. Let  $B\leq_R A$  be a f.g.   $R$-submodule. Since  $A$  is a semisimple   $R$-module, there exists  an  $R$-submodule  $C$  such that  $A = B \oplus C$        (see \cite[Corollary 4.3]{KOS2007}) with semisimple components   $B$ and $C$  (see \cite[Corollary 4.3]{KOS2007}). Thus
\[
A =B \oplus C= (\oplus_{j=1}^s B_j) \oplus  (\oplus_{j=1}^t  C_j),
\]
where each $B_j$ and $C_j $  are simple  $R$-submodules. Finally, according to  the Krull-Remak-Schmidt theorem  (see \cite[Chapter 6, Theorem 1.6]{CP1991}),  each   homogeneous component of  $B$  has  width at most  $r$, so $d_R(B) \leq  r$ by  part (iii).\end{proof}

\begin{lemma}\label{LeM:5}
Let   $A$  be an  $FG$-module, where  $FG$ is    the group algebra of a finite group $G$ over a field $F$.
If  $A$  has finite $r$-generator property, then   $\dim_F(A) \leq  r|G|$.
\end{lemma}

\begin{proof}
Let  $0\not=a\in A$. Clearly,  $\gp{a}_{FG}\cong_F FG/Ann_{FG}^{left}(a)$ and
\[
\dim_{FG}(\gp{a}_{FG})= \dim_F((FG)a) \leq \dim_F(FG)= |G|.
\]
Let  $B_1\leq_{FG} A$ such that $B_1$ is  f.g. and  $B_1\not=A$. Since  $A$  has $r$-generator property, there exists a  finite subset  $M = \{ a_1, \ldots, a_m \}$ such that  $B_1= \gp{M}_{FG}$   and  $m \leq  r$. Thus
\[
\dim_F(B) \leq  \dim_F((FG)a_1) + \cdots+ \dim_F((FG)a_m) \leq  m|G| \leq  r|G|.
\]
If  $0\not=d \in A\setminus   B$, then   $B_2:= B_1 + (FG)d$  is a  f.g. module,   such that  $B_1 \leq B_2$ but  $B_1 \not= B_2$.
Similarly, continuing this process, we construct a family of proper submodules $\{B_k < A\mid k\in\mathbb{N}\}$ such that each $\dim_F(B_k) \leq  r|G|$. Consequently, there exists   $m\in\mathbb{N}$  such that  $B_m = A$, so $A$  is a f.g. module and   $\dim_F(A) \leq  r|G|$.\end{proof}

The intersection of all maximal $R$-submodules of an $R$-module $A$ over a ring  $R$  is called the {\it Frattini submodule} of  $A$  and denoted by    $Fratt_R(A)$.  For finitely generated modules, the intersection of all
of maximal submodules is also called the radical of the module
by analogy with the Jacobson radical for rings. Recently, however, some other radicals have appeared in modules,
therefore we followed the example of Hartley \cite{Hartley} and called the intersection of all maximal submodules the Frattini submodule.

We freely use the fact  that $Fratt_R(A)$ consists of non-generating elements of $A$.

\begin{lemma}\label{LeM:6}
Let  $A=\gp{M}$  be an  $R$-module over  a ring $R$.  If  $S$  is a finite subset of $M$  such that    $S\subseteq  Fratt_R(A)$, then $A=\gp{M \setminus S}$.
\end{lemma}
\begin{proof}
First, assume that $S=\{d\}$.  Suppose the contrary, let the  subset  $M \setminus  \{d\}$  generate a proper  $R$-submodule  $B$. Clearly  $d\not\in  B$. Let  $\mathfrak{S}(B)$  be a family of all  $R$-submodules of  $A$  such that each member of  $\mathfrak{S}(B)$ includes  $B$  and does not contain the element $d$. Clearly,   $B\subseteq \mathfrak{S}(B)\not=\emptyset$ and  $\mathfrak{S}(B)$  has a maximal element  $C$ by   Zorn's lemma. If  $D$  is an  $R$-submodule of  $A$  such that  $C \leq  D$  and       $C \not= D$, then  $d\in D$. However,  in this case  $M\leq D$, which implies that  $D = A$. Consequently,   $C$  is a maximal  $R$-submodule of  $A$. On the other hand, $d\not\in   C$, so  we get a contradiction with a choice of $d$, which  proves the  result. The rest follows by induction on $|S|$.
\end{proof}

\begin{lemma}\label{LeM:7}
Let  $A$  be a Noetherian  $R$-module over a ring $R$.  If  $M\subseteq A$  such that  $A/Fratt_R(A)=\gp{ a + Fratt_R(A) \mid  a \in  M}_R$, then  $A=\gp{M}_R$ and
\[
d_R(A) = d_R(A/Fratt_R(A)).
\]
\end{lemma}

Note that,  outwardly, our  Lemma \ref{LeM:7} is very similar to the following  consequence of the Nakayama's lemma \cite[Corollary 4.8(b), p.\,124]{Eisenbud}. In the Nakayama's lemma  \cite[Corollary 4.8(b), p.\,124]{Eisenbud},   was  considered  a submodule,  obtained by acting on the entire module by the radical of the ring. However in our Lemma \ref{LeM:7} we  consider  a factor-module by the radical of the module, i.e. at the intersection of all maximal submodules. We provide the following  simple example:

Let $A$  be  a cyclic group of order $p> 2$ which is considered as a module over the ring of integers $\mathbb{Z}$. The radical (or the Frattini submodule)  $Fratt_\mathbb{Z}(A)$ of the module $A$ is a subgroup of order $p$ and the radical $\mathcal{J}(\mathbb{Z})$ of the ring of integers $\mathbb{Z}$ is zero.  Consequently, $\mathcal{J}(\mathbb{Z}) A=0$.

\begin{proof}[Proof of Lemma \ref{LeM:7}]
Let $B=\gp{M}_R\leq_RA$ such that $A\not=B$. The choice of  $M$  shows that  $A = B + D$. Let    $B_1\leq_R A$ be  generated by $M \cup  \{a_1\}$, where $a_1\in D\setminus  B$ and $B_1\not=A$.
Set $B_2=\gp{M \cup  \{a_1\} \cup  \{a_2\}}_R$ such that   $a_2 \in  D\setminus   B_1$.  If  $B_2 \not= A$, then we repeat  arguments above. This process can not be infinite because   $A$  is a Noetherian  $R$-module, so  there exist $a_1, \ldots, a_n \in  D$  such that   $\gp{M \cup  \{ a_1, \ldots, a_n \}}_R=A$ and  $A=\gp{M}_R$.

Let $M = \{ a_1, \ldots, a_k \}\subseteq A$, such that $k\in\mathbb{N}$ is minimal with the properties  $A=\gp{M}_R$ and $k = d_R(A)$. Clearly, $A/D =\gp{a_1 + D, \ldots, a_k + D}_R$, where $D = Fratt_{FG}(A)$. It follows that  $d_R(A/D) \leq  d_R(A)$. On the other hand, let  $\{ x_1, \ldots, x_m \}$  be a finite  subset of  $A$  such that  $A/D=\gp{ x_1 + D, \ldots, x_m + D }_R$. Thus   $A=\gp{x_1, \ldots, x_m}_R$. Consequently,   $m  \geq  d_R(A)$ and  $d_R(A/D) \geq d_R(A)$, so that  $d_R(A/D) = d_R(A)$.
\end{proof}

\begin{lemma}\label{LeM:8}
Let  $FG$ be  the  group algebra of a finite group $G$ over  a field $F$.
Let  $A$  be a f.g.   $FG$-module. If  $d_{FG}(A) = k$, then   $A/Fratt_{FG}(A)$  is semisimple of  finite width.  Moreover, each homogeneous component has width at most  $k$  and there exists a homogeneous component, whose width is exactly  $k$.
\end{lemma}

\begin{proof}
Let $M = \{ a_1, \ldots, a_k \}\subseteq A$ such that $k\in\mathbb{N}$ is  minimal with the property $A=\gp{M}_{FG}$. Clearly, each
$\gp{a_j}_{FG}\cong FG/Ann_{FG}^{left}(a_j)$, $\dim_F((FG)a_j) \leq  |G|$ and
\[
\dim_F(A) = \dim_F(FGa_1 + \cdots+ FGa_k)\leq k|G|.
\]
Since   $\dim_F(A)$ is finite, $A$  is a Noetherian   $FG$-module.
Let  $\{ V_1, \ldots, V_s \}$  be a collection  of maximal  $FG$-submodules of  $A$.  Let  $V = V_1 \cap \cdots\cap V_s$. Using the Remak's theorem we obtain an embedding of $A/V$   into a direct sum     $A/V_1 \oplus \cdots\oplus A/V_s$, in which each  factor-module  $A/V_j$  is simple. Since each  submodule of a semisimple module is itself semisimple  (see \cite[ Corollary 4.4]{KOS2007}), $A/V$   is a semisimple  $FG$-module. It follows that  $Fratt_{FG}(A) \leq  V$.

Let   $0\not=b\in V\setminus  Fratt_{FG}(A)$. Obviously,     $A$ contains a maximal  $FG$-submodule  $W$  such that  $b \not\in  W$. Put  $U = V \cap W$, then  A/U  is embedded into  $A/V \oplus A/W$. Using the above arguments, we obtain that  $A/U$  is a semisimple  $FG$-module, so that again  $Fratt_{FG}(A) \leq  U$. In addition, the choice of  $U$  implies that  $U \not= V$, therefore  $\dim_F(U) < \dim_F(V)$. If  $Fratt_{FG}(A) \not= U$, then we repeat the above argument. This process should be finite, because  $\dim_F(A)$ is finite. Consequently,   $A/Fratt_{FG}(A)$  is semisimple of  finite width.

Finally  $d_R(A/Fratt_{FG}(A)) = d_R(A) = k$ by Lemma \ref{LeM:7}  and  use Corollary \ref{Cor:6}. \end{proof}

Consider now the case when a Sylow  $p$-subgroup of  $G$  is normal.

\begin{lemma}\label{LeM:9}
Let  $FG$ be  the  group algebra of a finite $p$-group $G$   over  a field $F$ of prime characteristic  $p$.
The length of a $G$-central series of an $FG$-module   $A$ is finite and bounded by $|G|$. \end{lemma}

\begin{proof}
A  semidirect product   $A \rtimes  G$ is a nilpotent group (see  \cite[Lemma 3.8, p.\, 228]{BG1959}).   In particular,      $A \rtimes  \gp{ g }$   is also nilpotent for each element  $g \in  G$.

Set $A_0:= A$ and  $A_i: = (g-1)^iA$ for $i\geq 1$.  It is easy to see that
\[
\begin{split}
A_k &\not= \gp{0},  \quad   (g-1)A_k \not= A_k, \qquad\qquad (0\leq k\leq |g|-1)\\
A_{|g|} &= (g-1)^{|g|}A = (g^{|g|}-1)A = \gp{0}.
\end{split}
\]
Now we use induction on  $|G|$. If  $|G|=p$, then  $G=\gp{g}$  is  cyclic and
\begin{equation}\label{EFFsR}
A = A_0 \geq A_1 \geq \cdots\geq A_{|g|} = \gp{0}
\end{equation}
is a $G$-central series.

Assume now $|G| > p$. Since  a finite  $p$-group is nilpotent,   $\zeta(G) \not= \gp{1}$ and there exists  $1 \not= g \in  \zeta(G)$ such that  the  series \eqref{EFFsR}
is  $\gp{g}$-central, so that  $g \in  \mathrm{C}_G(A_j/A_{j + 1})$, where $\mathrm{C}_G(X)$ is the centralizer of a set $X$ in $G$.  An inclusion $ g \in  \zeta(G)$  shows that  each $A_j$  is an  $FG$-submodule of  $A$. Consequently, each  factor  $A_j/A_{j + 1}$ can be considered  as an  $F(G/\gp{g})$-module. Using the inequality    $|G/\gp{g}| <|G|$ (because  $g \not= 1$), we can apply the  induction hypothesis for  $A_j/A_{j + 1}$. By the induction  hypothesis, $A_j/A_{j + 1}$  has a  $(G/\gp{g})$-central series whose length is at most  $|G/\gp{g}|$. Obviously,  factors of such series are  also  $G$-central for each  $0 \leq j \leq k-1$. Thus   $A$  has a $G$-central series whose length is at most  $|g| |G/\gp{g}|= |G|$.
\end{proof}

\section{Proofs}
\begin{proof}[Proof of Theorem \ref{Th:1}]
Let $A$  be an  $FG$-module in which $FG$ is the group algebra of a finite group $G$ over the   field $F$. Since  $A$  has finite $r$-generator property,  $A$  has  finite $F$-dimension by Lemma \ref{LeM:5}, so $A$ has a finite series of  $FG$-submodules, whose factors are semisimple.   Let  $B\leq_{FG} A$  such that  $d_{FG}(B) = r$ and   $D = Fratt_{FG}(B)$. Corollary \ref{LeM:7} implies that  $d_{FG}(B) = d_{FG}(B/D) = r$ and the  factor-module   $B/D$  is semisimple and it has  finite width by Lemma   \ref{LeM:8}.  Moreover, every homogeneous component has width at most  $r$  and there exists a homogeneous component, whose width is exactly  $r$.

Let  $V/W$  be a semisimple homogeneous  $FG$-factor and set $U:= Fratt_{FG}(V)$. Obviously,   $U \leq  W$ and  the inequality   $d_{FG}(V) \leq  r$  implies that  $d_{FG}(V/U) \leq  r$ by Corollary \ref{LeM:7}.  Consequently,  $V/U$  is semisimple   by Lemma   \ref{LeM:8} and each of its homogeneous components has width at most  $r$. The isomorphism
\[
V/W \cong_{FG}(V/U)/(W/U)
\]
implies that  $V/W$  has  width at most  $r$.

Conversely, let  $C/E$  be a semisimple homogeneous  $FG$-factor, having width  $r$. Thus  $Fratt_{FG}(C) \leq  E$, so $C/Fratt_{FG}(C)$  has a homogeneous factor whose width is exactly  $r$. By the statements  of our  theorem,  each  another homogeneous factor of  $C/Fratt_{FG}(C)$  has  width at most  $r$. Then $d_{FG}(C/Fratt_{FG}(C)) = r$ by Lemma   \ref{LeM:8}, and    $d_{FG}(C) = r$ by Corollary \ref{LeM:7}. Finally, let  $S$  be an $FG$-submodule of  $A$. The factor  $S/Fratt_{FG}(S)$  is semisimple  by Lemma   \ref{LeM:8}  and  each  homogeneous factor of  $S/Fratt_{FG}(S)$  has  width at most  $r$ by  parts (i)-(ii)  of our  theorem.   Finally, using Lemma   \ref{LeM:8}, we obtain that  $d_{FG}(S/Fratt_{FG}(S)) \leq  r$ and $d_{FG}(S) \leq  r$, so  $A$  has $r$-generator property.
\end{proof}

\begin{proof}[Proof of Theorem \ref{Th:2}]
Let  $F$  be a field of prime characteristic  $p$  and  let $G$  be a finite group, such   that   $p$  does not divide  $|G|$. If  $A$  is an  $FG$-module, then  $A$  is semisimple  (see \cite[Corollary 5.15]{KOS2007}). It follows that the  group ring $FG$ is semisimple. Denote by  $nns_F(G)$  the number of pairwise non-isomorphic simple  $FG$-modules. This number is finite  (see \cite{BS1952, BS1956})  and can be  calculated  in  the following way.

Let $\xi$ be  a primitive $n$-th root of unity, where  $n$ is the greatest divisor of $\exp(G)$ which is not divisible by $char(F)$. Two  elements  $a, b \in G$  are  called  {\it $F$-conjugate} if  there exist  $x \in  G$  and   $m\in \mathbb{N}$  such that  $x^{-1}ax =b^m$ and  the map  $\xi\mapsto \xi^m$  is extensible to an automorphism of the field  $F(\xi)$  fixing  the subfield  $F$  elementwise. Berman \cite{BS1956} and Witt \cite{Witt}  proved that  $\mathrm{nns}_F(G)$  is equal to the number of $F$-conjugate classes of  $F$-regular elements of   $G$ (see \cite[Theorem 42.8 p.\,306]{Curtis_Reiner}).   If  $F$ is   algebraically closed, then  Brauer \cite{BR1956, Curtis_Reiner} shows that  $\mathrm{nns}_F(G)$  is equal to the number of the conjugate classes of  elements of   $G$ (see \cite[Theorem 40.1 p.\,283]{Curtis_Reiner}).  In both  cases   $\mathrm{nns}_F(G)$ is bounded by the number of conjugacy classes of $G$.

Since    $\dim_F(A)$ is  finite,    $A =\oplus_{j=1}^n H_j$ in which  each  $H_j$  is a homogeneous component of  $A$. As we noted above  $n \leq nns_F(G)$. The fact that  $H_j$  is a direct sum of at most  $r$  simple submodules follows from Corollary \ref{Cor:6}(i-iii). \end{proof}

\begin{proof}[Proof of Theorem \ref{Th:3}]
Clearly,    $\dim_F(A)$ is  finite  by  Lemma \ref{LeM:5} and   $A$  has a series of  $FG$-submodules (see  Lemma \ref{LeM:9})
\[
\gp{0} = Z_0 \leq   Z_1 \leq  \cdots\leq  Z_{m-1} \leq  Z_m = A
\]
whose factors are  $p$-central, that is  $P \leq  \mathrm{C}_G(Z_{j + 1} /Z_j)$ for each  $0 \leq  j \leq  m-1$. It follows that  each $G/\mathrm{C}_G(Z_{j + 1} /Z_j)$  is a  $p$-group, so we can apply Theorem \ref{Th:2} to each of these factors.
\end{proof}

\bibliographystyle{abbrv}

\begin{thebibliography}{10}

\bibitem{BG1959}
G.~Baumslag.
\newblock Wreath products and {$p$}-groups.
\newblock {\em Proc. Cambridge Philos. Soc.}, 55(3):224--231, 1959.

\bibitem{BS1952}
S.~D. Berman.
\newblock On the theory of representations of finite groups.
\newblock {\em Doklady Akad. Nauk SSSR (N.S.)}, 86:885--888, 1952.

\bibitem{BS1956}
S.~D. Berman.
\newblock The number of irreducible representations of a finite group over an
  arbitrary field.
\newblock {\em Dokl. Akad. Nauk SSSR (N.S.)}, 106:767--769, 1956.

\bibitem{Bovdi_book}
A.~A. Bovdi.
\newblock Group rings. ({R}ussian).
\newblock {\em Kiev.UMK VO}, page 155, 1988.

\bibitem{Bovdi_Kurdachenko}
V.~A. Bovdi and L.~A. Kurdachenko.
\newblock Some ranks of modules over group rings.
\newblock {\em Comm. Algebra}, 49(3):1225--1234, 2021.

\bibitem{BR1956}
R.~Brauer.
\newblock Zur {D}arstellungstheorie der {G}ruppen endlicher {O}rdnung.
\newblock {\em Math. Z.}, 63:406--444, 1956.

\bibitem{CI1950}
I.~S. Cohen.
\newblock Commutative rings with restricted minimum condition.
\newblock {\em Duke Math. J.}, 17:27--42, 1950.

\bibitem{CP1991}
P.~M. Cohn.
\newblock Algebra. {V}ol. 3.
\newblock pages xii+474, 1991.

\bibitem{Curtis_Reiner}
C.~W. Curtis and I.~Reiner.
\newblock {\em Methods of representation theory. {V}ol. {I}}.
\newblock Wiley Classics Library. John Wiley \& Sons, Inc., New York, 1990.
\newblock With applications to finite groups and orders, Reprint of the 1981
  original, A Wiley-Interscience Publication.

\bibitem{Giovanni_2013}
F.~de~Giovanni.
\newblock Infinite groups with rank restrictions on subgroups.
\newblock {\em Zap. Nauchn. Sem. S.-Peterburg. Otdel. Mat. Inst. Steklov.
  (POMI)}, 414(Voprosy Teorii Predstavleni\u{\i} Algebr i Grupp. 25):31--39,
  2013.

\bibitem{DKS2017}
M.~R. Dixon, L.~A. Kurdachenko, and I.~Y. Subbotin.
\newblock {\em Ranks of groups}.
\newblock John Wiley \& Sons, Inc., Hoboken, NJ, 2017.
\newblock The tools, characteristics, and restrictions.

\bibitem{Eisenbud}
D.~Eisenbud.
\newblock {\em Commutative algebra}, volume 150 of {\em Graduate Texts in
  Mathematics}.
\newblock Springer-Verlag, New York, 1995.
\newblock With a view toward algebraic geometry.

\bibitem{FS2001}
L.~Fuchs and L.~Salce.
\newblock {\em Modules over non-{N}oetherian domains}, volume~84 of {\em
  Mathematical Surveys and Monographs}.
\newblock American Mathematical Society, Providence, RI, 2001.

\bibitem{GR1969}
R.~Gilmer.
\newblock Two constructions of {P}r\"{u}fer domains.
\newblock {\em J. Reine Angew. Math.}, 239/240:153--162, 1969.

\bibitem{GR1972}
R.~Gilmer.
\newblock On commutative rings of finite rank.
\newblock {\em Duke Math. J.}, 39:381--383, 1972.

\bibitem{GR1973}
R.~Gilmer.
\newblock The {$n$}-generator property for commutative rings.
\newblock {\em Proc. Amer. Math. Soc.}, 38:477--482, 1973.

\bibitem{Hartley}
B.~Hartley.
\newblock A class of modules over a locally finite group. {III}.
\newblock {\em Bull. Austral. Math. Soc.}, 14(1):95--110, 1976.

\bibitem{Heitmann_1976}
R.~C. Heitmann.
\newblock Generating ideals in {P}r\"{u}fer domains.
\newblock {\em Pacific J. Math.}, 62(1):117--126, 1976.

\bibitem{KOS2007}
L.~A. Kurdachenko, J.~Otal, and I.~Y. Subbotin.
\newblock {\em Artinian modules over group rings}.
\newblock Frontiers in Mathematics. Birkh\"{a}user Verlag, Basel, 2007.

\bibitem{MAL1948}
A.~I. Mal'cev.
\newblock On groups of finite rank.
\newblock {\em Mat. Sbornik N.S.}, 22(64):351--352, 1948.

\bibitem{Matson_2009}
A.~Matson.
\newblock Rings of finite rank and finitely generated ideals.
\newblock {\em J. Commut. Algebra}, 1(3):537--546, 2009.

\bibitem{Pettersson_1999}
K.~Pettersson.
\newblock Strong {$n$}-generators and the rank of some {N}oetherian
  one-dimensional integral domains.
\newblock {\em Math. Scand.}, 85(2):184--194, 1999.

\bibitem{Swan_1984}
R.~G. Swan.
\newblock {$n$}-generator ideals in {P}r\"{u}fer domains.
\newblock {\em Pacific J. Math.}, 111(2):433--446, 1984.

\bibitem{Witt}
E.~Witt.
\newblock Die algebraische {S}truktur des {G}ruppenringes einer endlichen
  {G}ruppe \"{u}ber einem {Z}ahlk\"{o}rper.
\newblock {\em J. Reine Angew. Math.}, 190:231--245, 1952.

\end{thebibliography}

\end{document}